\documentclass[12pt]{amsart}
\usepackage{amsfonts}
\usepackage{amsmath}
\usepackage{amssymb}
\usepackage{graphicx}
\newtheorem{theorem}{Theorem}[section]

\newtheorem{corollary}[theorem]{Corollary}

\theoremstyle{definition}

\theoremstyle{remark}
\newtheorem{remark}[theorem]{Remark}
\numberwithin{equation}{section}

\begin{document}
\title{On Borel Summability and Analytic Functionals}
\author{Ricardo Estrada}
\address{Department of Mathematics, Louisiana State University, Baton Rouge, LA 70803, USA.}
\email{restrada@math.lsu.edu}
\author{Jasson Vindas}
\address{Department of Mathematics\\ Ghent University\\ Krijgslaan 281 Gebouw S22\\ B 9000 Gent\\ Belgium}
\email{jvindas@cage.Ugent.be}
\thanks{R. Estrada gratefully acknowledges support from NSF, through grant number
0968448}
\thanks{J. Vindas gratefully acknowledges support by a Postdoctoral Fellowship of the Research Foundation--Flanders (FWO, Belgium)}
\begin{abstract}
We show that a formal power series has positive radius of convergence if and only if it is uniformly Borel summable over a circle with center at the origin. Consequently, we obtain that an entire function $f$ is of exponential type if and only if the formal power series $\sum_{n=0}^{\infty}f^{(n)}(0)z^{n}$ is uniformly Borel summable over a circle centered a the origin. We apply these results to obtain a characterization of those Silva tempered ultradistributions which are analytic functionals. We also use Borel summability to represent analytic functionals as Borel sums of their moment Taylor series over the Borel polygon.  
\end{abstract}

\maketitle
\section{Introduction}
\label{introBorel}
The aim of this note is to characterize those entire functions $f$ which are of exponential type in terms of the Borel summability of the formal power series

\begin{equation}
\label{Boreleq1}
\sum_{n=0}^{\infty}f^{(n)}(0)z^{n}\ .
\end{equation}

We first prove that a formal power series has positive radius of convergence if and only if it is uniformly $(\mathrm{B}')$ summable over some circle with center at the origin. Observe that this is the converse to Borel's classical theorem \cite{borel,hardy,korevaarbook}: If a formal power series has positive radius of convergence, then it is uniformly Borel summable on compacts inside the disk of convergence. Apparently, such a converse result has not been given elsewhere before. 

We then obtain the desired characterization for entire functions of exponential type, we show that if (\ref{Boreleq1}) is uniformly $(\mathrm{B}')$ summable over some circle centered at the origin, then $f$ is of exponential type. This characterization yields a characterization of those Silva tempered ultradistributions \cite{hoskins-pinto} which are analytic functionals, that is, elements of $\mathcal{O}'(\mathbb{C})$, the dual of the space of entire functions \cite{morimoto}. These results will be presented in Section \ref{Borel} below. In Section \ref{Borelpolygon}, we make some comments about the representation of analytic functionals by the Borel sums of the moment Taylor series at infinity associated to the functional. 

\section{Preliminaries}
\label{preliBorel}
A series $\sum_{n=0}^{\infty}a_{n}$ is said \cite{borel,hardy,korevaarbook} to be summable to a number $\beta$ by the Borel method $(\mathrm{B}')$ if the power series $f(t)=\sum_{n=0}^{\infty}(a_n/n!)t^{n}$ is convergent for any $t$ and
\begin{equation*}
%\label{preliBoreleq1}
\int_{0}^{\infty^{-}}e^{-t}f(t)\mathrm{d}t=\lim_{\lambda\to\infty}\int_{0}^{\lambda}e^{-t}f(t)\mathrm{d}t=\beta\ ,
\end{equation*}
in such a case one writes
\begin{equation*}
%\label{preliBoreleq2}
\sum_{n=0}^{\infty}a_{n}=\beta \ \ \ (\mathrm{B}')\ .
\end{equation*}
Suppose that the formal power series 
\begin{equation}
\label{preliBoreleq3}
\sum_{n=0}^{\infty}a_{n}z^{n} 
\end{equation}
is $(\mathrm{B}')$ summable at a point $z_{0}\in\mathbb{C}$, then \cite[Thm. 130]{hardy} it is uniformly $(\mathrm{B}')$ summable on the line segment $[0,z_{0}]$ and \cite[Thm. 132]{hardy} its sum extends to an analytic function inside the circle with diameter $[0,z_{0}]$, that is, the disk $\Delta_{z_{0}}:=\left|z-z_{0}/2\right|<\left|z_{0}\right|/2$. Explicitly, the analytic extension is given by
\begin{equation*}
%\label{preliBoreleq4}
b_{z_{0}}(z)=\frac{z_{0}}{z}\int_{0}^{\infty^{-}}e^{-\frac{z_{0}}{z}t}f(z_{0}t)\mathrm{d}t\ ,
\end{equation*}
with $f$ as before $f(t)=\sum_{n=0}^{\infty}(a_n/n!)t^{n}$, so that 
\begin{equation*}
%\label{preliBoreleq5}
b_{z_{0}}(z)=\sum_{n=0}^{\infty}a_{n}z^{n}\ \ \ (\mathrm{B}')\ , \ \ z\in(0,z_{0})\ .
\end{equation*}
In general, it is not true that the formal power series (\ref{preliBoreleq3}) should be $(\mathrm{B}')$ at any other point of $\Delta_{z_{0}}$; indeed, consider $a_{n}=\sum_{k=0}^{\infty}(-1)^{k}k^{n}/k!$, then (\ref{preliBoreleq3}) is $(\mathrm{B}')$ summable on the real axis but not at any other point of the half-plane $\Re e\: z>0$ \cite[p. 190]{hardy}. On the other hand, different $b_{z_{0}}$ may lead to different analytic functions in the regions of summability, an example of this can be found in \cite[p. 189]{hardy}. 

We shall use some concepts from the theory of Silva tempered ultradistributions \cite{hoskins-pinto,ssilva1,ssilva2}, for which we follow the notation from \cite{hoskins-pinto}. For analytic functionals, we use the notation from \cite{morimoto}. Let $\mathcal{U}(\mathbb{C})$ be the space of entire functions which are rapidly decreasing on horizontal strips, that is, $\varphi\in\mathcal{U}(\mathbb{C})$ if
\begin{equation}
\label{preliBoreleq6}
\left\|\varphi\right\|_{k}:=\operatorname*{sup}_{\left|\Im m\:s\right|\leq k}(1+\left|s\right|^{k})\left|\varphi(s)\right|<\infty\ .
\end{equation}
The seminorms (\ref{preliBoreleq6}) give to $\mathcal{U}(\mathbb{C})$ a Fr\'{e}chet space structure. The dual space $\mathcal{U}'(\mathbb{C})$ is called the space of Silva tempered ultradistributions. 

Every Silva tempered ultradistribution admits an analytic representation in the following sense; given $g\in\mathcal{U}'(\mathbb{C})$ there exists $G$ analytic in some region of the form $\left|\Im m\:s\right| >c>0$ and having at most polynomial growth such that
\begin{equation*}
%\label{preliBoreleq7}
\left\langle g,\varphi\right\rangle= -\oint_{\Gamma_{c'}}G(s)\varphi(s)\mathrm{d}s\ , \ \ \forall \varphi\in\mathcal{U}(\mathbb{C})\ ,
\end{equation*}
where $\Gamma_{c'}$ is the counterclockwise oriented boundary of the strip $\left|\Im m\:s\right| \leq c'$, where $c'>c$. Analytic representations are unique modulo polynomials. 

The functional $g\in\mathcal{U}'(\mathbb{C})$ is said to vanish on a \emph{real} open interval $I\subset\mathbb{R}$ if $G$ has an analytic continuation on the vertical strip $\Re e\:s\in I$. The \emph{real support} of a Silva tempered ultradistribution $g$ is then the complement (in $\mathbb{R}$) of the largest open set on which it vanishes, we denote it by $\operatorname*{supp}_{\mathbb{R}}g$. The subspace of Silva tempered ultradistributions with compact real support coincides (under a natural identification) with the space of analytic functionals $\mathcal{O}'(\mathbb{C})$. Any element of $\mathcal{O}'(\mathbb{C})$ can be realized in a unique manner as the germ of a function which is analytic outside a disk and vanishes at $\infty$.

Finally, we remark that $\mathcal{U}'(\mathbb{C})$ is isomorphic, via Fourier transform, to $\mathcal{K}'_{1}(\mathbb{R})$, the space of distributions of exponential type, that is, the ones which are derivatives of continuous functions that grow not faster than $e^{\alpha\left|t\right|}$, for some $\alpha\in\mathbb{R}$. We fix the constants in the Fourier transform so that
\begin{equation*}
%\label{preliBorel8}
\hat{\phi}(u)=\int_{-\infty}^{\infty}e^{iut}\phi(t)\mathrm{d}t\ ,
\end{equation*}
on test functions.

\section{On Borel Summability}
\label{Borel}
Our main result is the following theorem.

\begin{theorem}
\label{Borelth1}
Let $\sum_{n=0}^{\infty}a_{n}z^{n}$ be a formal power series. If for some $r>0$ the series $\sum_{n=0}^{\infty}a_{n}r^{n}e^{in\theta}$ is uniformly $(\mathrm{B}')$ summable for $\theta\in[0,2\pi]$, then $\sum_{n=0}^{\infty}a_{n}z^{n}$ has positive radius of convergence $R\geq r$.
\end{theorem}
\begin{proof} Set $f(t)=\sum_{n=0}^{\infty}a_{n}t^{n}/n!$. We may assume that $r=1$. We divide the proof into three steps.
\smallskip

\emph{First step.} We shall first show that each function
$$
b_{\theta}(z):=\frac{e^{i\theta}}{z}\int_{0}^{\infty^{-}}e^{-\frac{e^{i\theta}}{z}t}f(e^{i\theta}t)\mathrm{d}t\ , \ \ \theta\in[0,2\pi)\ ,
$$
analytic on the disk $\Delta_{\theta}:=\left|z-e^{i\theta}/2\right|<1/2$, can be analytically continued to $\mathbb{D}\setminus\left\{0\right\}:=\left|z\right|<1,\ z\neq0$. Furthermore, we show that there exists $b\in\mathcal{O}(\mathbb{D}\setminus\left\{0\right\})$ such that $b(z)=b_{\theta}(z)$ for all $\theta$ on their corresponding domains. 

Set $b(z)=b_{0}(z)$ on $\Delta_{0}$. Fix $\theta\in(0,\pi/2)$. We shall prove that $b(z)=b_{\theta}(z)$ on $\Delta_{0}\cap\Delta_{\theta}$. For this purpose, it is enough to show that $b(x)=b_{\theta}(x)$ for  $x\in\Delta_{0}\cap\Delta_{\theta}\cap\mathbb{R}=(0,\cos \theta)$. Thus, we must show that 
$$
xb_{\theta}(x)=e^{i\theta}\int_{0}^{\infty^{-}}e^{-\frac{e^{i\theta}}{x}t}f(e^{i\theta}t)\mathrm{d}t=\int_{0}^{\infty^{-}}e^{-\frac{t}{x}}f(t)\mathrm{d}t=xb(x)\ ,
$$
for $0<x<\cos\theta$. The Cauchy theorem applied to the contour given by the union of the line segments $[0,\lambda]$, $[0,\lambda e^{i\theta}]$
and the circular arc $\lambda e^{i\omega}$, $\omega\in[0,\theta]$, gives us
$$
\lim_{\lambda\to\infty}i\lambda \int_{0}^{\theta}e^{-\frac{\lambda}{x}e^{i\omega}}e^{i\omega}f(\lambda e^{i\omega})\mathrm{d}\omega=xb_{\theta}(x)-xb(x)\ .
$$
Setting
$$
\psi(\lambda)=\lambda\int_{0}^{\theta}e^{-\frac{\lambda}{x}e^{i\omega}}e^{i\omega}f(\lambda e^{i\omega})\mathrm{d}\omega,
$$
we now have to show that $\lim_{\lambda\to\infty}\psi(\lambda)=0$. It is easier if we calculate this limit (which a priori we know its existence!) in the Ces\`{a}ro sense. Set
$$
M_{\omega}(t)=\int_{0}^{t}e^{-u}f(e^{i\omega}u)\mathrm{d}u\ , \ \ \omega\in[0,\theta]\ ,
$$
since $\sum_{n=0}^{\infty}a_{n}e^{in\omega}$ is uniformly $(\mathrm{B}')$ summable, then $M_{\omega}(t)=O(1)$, as $t\to\infty$ uniformly in $\omega$. Thus, as $\lambda\to\infty$,
\begin{align*}
\left|\frac{1}{\lambda}\int_{0}^{\lambda}\psi(t)\mathrm{d}t\right|&
=\frac{1}{\lambda}\left|\int_{0}^{\theta}e^{iw}\int_{0}^{\lambda}te^{-\frac{t}{x}e^{i\omega}}f(e^{i\omega}t)\mathrm{d}t\mathrm{d}\omega\right|
\\
&
\leq\int_{0}^{\theta}e^{-\lambda\left( \frac{\cos \omega}{x}-1\right)}\left|M_{\omega}(\lambda)\right|\mathrm{d}\omega
\\
&
\ \ \ \ \ \ 
+\frac{1}{\lambda}\int_{0}^{\theta}\left|\int_{0}^{\lambda}\left(te^{-t\left(\frac{e^{i\omega}}{x}-1\right)}\right)'M_{\omega}(t)\mathrm{d}t\right|\mathrm{d}\omega
\\
&
=O\left(e^{-\lambda\left(\frac{\cos \theta}{x}-1\right)}\right)+O\left(\frac{1}{\lambda}\right)=o(1)\ .
\end{align*}
Therefore, $\lim_{\lambda\to\infty}\psi(\lambda)=0\ $ $(\mathrm{C},1)$, but since the ordinary limit coincides with the Ces\`{a}ro one, we get $b(x)=b_{\theta}(x)$, for $0<x<\cos\theta$. 

We have shown that $b(z):=b_\theta(z)$ for $z\in\Delta_{\theta}$, $0\leq\theta<\pi/2$, is well defined. Moreover, it provides us the desired analytic continuation beyond the first quadrant. Reasoning in the same way, we continue $b(z)$ to $\mathbb{D}\setminus\left\{0\right\}$ as an analytic function. 

\smallskip

\emph{Second step.} We now show that $b$ is actually analytic at the origin.

We first prove that $zb(z)\in\mathcal{O}(\mathbb{D})$. For that, it is enough to prove that $zb(z)$ is bounded in the set $\left\{z\neq0:\left|z\right|<1/2\right\}$. Let 

$$\Omega_{\theta}=\left\{z:  -\theta-\frac{\pi}{4}<\arg \left(\frac{1}{z}\right)<-\theta+\frac{\pi}{4} \right\}\cap\left\{z:\left|z-\frac{e^{i\theta}}{4}\right|<1/4\right\}\ ,$$
then $\Omega_{\theta}\subset\Delta_{\theta}$ and  $\left\{z\neq0:\left|z\right|<1/2\right\}=\bigcup\Omega_{\theta}$. Take now $z\in\Omega_{\theta}$, write $s=e^{i\theta}/z$, then $\Re e\: s>2$, $-\pi/4<\arg s<\pi/4$, and

\begin{align*}
\left|zb(z)\right|&=\left|zb_{\theta}(z)\right|=\left|\int_{0}^{\infty^{-}}e^{-st}f(e^{i\theta}t)\mathrm{d}t\right|
\\
&
\leq\left|s-1\right|\int_{0}^{\infty}e^{-t(\Re e\:s-1)}\left|M_{\theta}(t)\right|\mathrm{d}t
\\
&
\leq
\sec (\arg (s-1))\max_{t,\:\theta}\left|M_{\theta}(t)\right|
\\
&
\leq
\sqrt{5}\max_{t,\:\theta}\left|M_{\theta}(t)\right|\ .
\end{align*}
Consequently, $zb(z)$ is uniformly bounded for $\left|z\right|<1/2$, $z\neq0$; and so $zb(z)\in\mathcal{O}(\mathbb{D})$. 

Observe now that either $b$ is analytic at the origin or it has a pole of order 1. If $b$ has a pole, then $b(x)\sim A/x$, as $x\to0^{+}$ on the real axis, for some constant $A$. However, it is impossible because of \cite[Thm. 130]{hardy}; indeed, $\sum_{n=0}^{\infty}a_{n}x$ is uniformly $(\mathrm{B}')$ summable on $[0,1]$ which shows that $\lim_{x\to0^{+}}b(x)=a_{0}$. Thus, we have shown that $b\in\mathcal{O}(\mathbb{D})$.

\smallskip

\emph{Third step.} We now use the analyticity of $b$ to show that $f$ is of exponential type. Based on this fact we will show that the power series has a positive radius of convergence.

Observe that the $(\mathrm{B}')$ summability of the two series $\sum_{n=0}^{\infty}a_{n}$ and $\sum_{n=0}^{\infty}(-1)^{n}a_{n}$ gives us at once that $f\in\mathcal{K}'_{1}(\mathbb{R})$. We now calculate an explicit analytic representation of its Fourier transform, $\hat{f}\in\mathcal{U}'(\mathbb{C})$. It is not hard to see \cite{hoskins-pinto,ssilva1} that
$$
F(s)=\left\{ \begin{array}{ll}
\displaystyle\int_{0}^{\infty^{-}}e^{ist}f(t)\mathrm{d}t, & \Im m\:s>1\ ,\\
\displaystyle-\int^{0}_{-\infty^{+}}e^{ist}f(t)\mathrm{d}t, & \Im m\: s<-1\ ,\\
\end{array}\right.
$$
works. Notice that, for $\Im m\:s>1$,
$$
F(s)=\frac{i}{s}\:b_{0}\left(\frac{i}{s}\right)=\frac{i}{s}\:b\left(\frac{i}{s}\right)\ ,
$$
and, for $\Im m\:s<-1$,
$$
F(s)=\frac{i}{s}\:b_{\pi}\left(\frac{i}{s}\right)=\frac{i}{s}\:b\left(\frac{i}{s}\right)\ .
$$

Therefore, $(i/s)b(i/s)$ is an analytic continuation of $F(s)$ to $\mathbb{C}\setminus\overline{\mathbb{D}}$. So, the real support of the Silva tempered ultradistribution $\hat{f}$ is compact, $\operatorname*{supp}_{\mathbb{R}}\hat{f}\subseteq[-1,1]$; consequently, $\hat{f}\in\mathcal{O}'(\mathbb{C})$. The Caley-Wiener-Gel'fand representation theorem \cite[Thm. 4.21, p. 175]{hoskins-pinto} implies now that $f$ is an entire function of exponential type (with type less or equal to 1); actually, we obtain for free from the same theorem the realization of $\hat{f}$ as the Taylor series at infinity
$$
\sum_{n=0}^{\infty}\frac{a_{n}i^{n+1}}{s^{n+1}}\ ,
$$ 
convergent on $\mathbb{C}\setminus\overline{\mathbb{D}}$. Hence, $\sum_{n=0}^{\infty}a_{n}z^{n}$ has positive radius of convergence greater or equal to 1. This completes the proof.
\end{proof}

We immediately obtain the following two corollaries.

\begin{corollary}
\label{Borelc1}
Let $f\in\mathcal{O}(\mathbb{C})$. Then $f$ is of exponential type if and only if there exists $r>0$ such that $\sum_{n=0}^{\infty}f^{(n)}(0)r^{n}e^{in\theta}$ is uniformly $(\mathrm{B}')$ summable for $\theta\in[0,2\pi]$. In such a case, the type of $f$ is less or equal to $r^{-1}$.
\end{corollary}

\begin{corollary}
\label{Borelc2}
A Silva tempered ultradistribution $g\in\mathcal{U}'(\mathbb{C})$ is an analytic functional, i.e., $g\in\mathcal{O}'(\mathbb{C})$, if and only $\hat{g}\in\mathcal{O}(\mathbb{C})$ and there exists $r>0$ such that $\sum_{n=0}^{\infty}\hat{g}^{(n)}(0)r^{n}e^{in\theta}$ is uniformly $(\mathrm{B}')$ summable for $\theta\in[0,2\pi]$. Moreover, the analytic realization of $g$ vanishing at $\infty$ is analytic on $\mathbb{C}\setminus(r^{-1}\overline{\mathbb{D}})$.
\end{corollary}
\section{On Analytic Functionals and the Borel Polygon}
\label{Borelpolygon}

In connection to Corollary \ref{Borelc2}, it seems appropriate to make some comments about analytic representations of analytic functionals and Borel polygons.

Given $g\in\mathcal{O}'(\mathbb{C})$, it can always be represented as a series of multipoles \cite{ssilva2},

\begin{equation*}
%\label{Borelpolygoneq1}
g=\sum_{n=0}^{\infty}\frac{(-1)^{n}\mu_{n}}{n!}\:\delta^{(n)}\ ,
\end{equation*}
where $\mu_{n}=\mu_{n}(g):=\left\langle g,s^{n}\right\rangle$ are the moments of $g$ and $\delta$ is the Dirac delta. In addition, we have that $b(z)=\sum_{n=0}^{\infty}\mu_{n}z^{n}$ has positive radius of convergence; we use the notation $\Pi(b)$ for the usual Borel polygon \cite{borel,hardy,korevaarbook} (with base point 0) of $b$. We define the \emph{Borel polygon of the analytic functional} $g$ as

\begin{equation*}
%\label{Borelpolygoneq2}
\Pi(g)=\left\{s\in\mathbb{C}:\frac{1}{s}\in\Pi(b)\right\}\ .
\end{equation*}
Then, it is clear from the usual result for the Borel polygon of power series \cite[Thm. 4.3, p. 286]{korevaarbook} that the formal series

$$-\frac{1}{2\pi i} \sum_{n=0}^{\infty}\frac{\mu_{n}}{s^{n+1}}$$
is $(\mathrm{B'})$ summable on $\operatorname*{int}\Pi(g)$, uniformly $(\mathrm{B'})$ summable on any closed subset which is away from the boundary of $\Pi(g)$, and not $(\mathrm{B'})$ summable outside the closure of $\Pi(g)$; moreover, the Borel sum is analytic on $\operatorname*{int}\Pi(g)\cup\left\{\infty\right\}$. It is also easy to verify that $\mathbb{C}\setminus\operatorname*{int}{\Pi(g)}$ is simply connected and compact. Hence, by the Cauchy theorem, we obtain the ensuing theorem.

\begin{theorem}
Let $g\in\mathcal{O}'(\mathbb{C})$. Then 
\begin{equation}
\label{Borelpolygoneq3}
G(s)=-\frac{1}{2\pi i} \sum_{n=0}^{\infty}\frac{\mu_{n}}{s^{n+1}}\ \ \  (\mathrm{B}')
\end{equation}
is analytic in the interior of $\Pi(g)$ and for any $\varphi\in\mathcal{O}(\mathbb{C})$ we have
\begin{equation}
\label{Borelpolygoneq4}
\left\langle g,\varphi\right\rangle=-\oint_{\mathsf{C}} G(s)\varphi(s)\mathrm{d}s\ ,
\end{equation}
where $\mathsf{C}$ is any contour contained in the interior of the Borel polygon $\Pi(g)$ which winds once around $\mathbb{C}\setminus\operatorname*{int}{\Pi(g)}$.
\end{theorem}

\begin{remark}
Observe that using \cite[Thm. 2.2.1]{morimoto} we obtain that $g\in\mathcal{O}'(\mathbb{C}\setminus\operatorname*{int}{\Pi(g)})\subset \mathcal{O}'(\mathbb{C})$, where $\mathcal{O}(\mathbb{C}\setminus\operatorname*{int}{\Pi(g)})$ is the space of germs of analytic functions in a neighborhood of $\mathbb{C}\setminus\operatorname*{int}{\Pi(g)}$. The Sebasti\~{a}o e Silva-K\"{o}the-Grothendieck theorem \cite[Thm. 2.13]{morimoto} provides us also with a representation of the form (\ref{Borelpolygoneq4}) for $\varphi\in\mathcal{O}(\mathbb{C}\setminus\operatorname*{int}{\Pi(g)})$ , where $G$ is again given by the Borel sum (\ref{Borelpolygoneq3}).
\end{remark}

\end{document}